\documentclass[11pt, leqno]{amsart}
\usepackage{amssymb,amscd,amsthm,amsxtra}
\usepackage{latexsym}
\usepackage[mathscr]{euscript}
\usepackage{amsfonts}
\usepackage{color}
\usepackage[colorlinks=true]{hyperref}
\usepackage{mathrsfs}
\makeatletter
\newtheorem{theorem}{Theorem}[section]

\newtheorem{proposition}[theorem]{Proposition}

\numberwithin{equation}{section}
\begin{document}
%%%%%%%%%%%%%%%%%%%%%%%%%%%%%%%%%%%%%%
%%%%%%%%%%%%%%%%%%%%%%%%%%%%%%%%%%%%%%
%%%%%%%%%%%%%%  PREAMBLE  %%%%%%%%%%%%%%%%
%%%%%%%%%%%%%%%%%%%%%%%%%%%%%%%%%%%%%%
%%%%%%%%%%%%%%%%%%%%%%%%%%%%%%%%%%%%%%
\title[Semigroups associated with the Laplace operator on homogeneous trees]{Pointwise convergence of the heat and subordinates of the heat semigroups associated with the Laplace operator on homogeneous trees and two weighted $L^p$ maximal inequalities.}
\author[Isaac Alvarez-Romero]{I. Alvarez-Romero $^{\dag,1}$}
\author[Bego\~na Barrios]{B. Barrios$^{\ddag,2}$}
\author[Jorge. J. Betancor]{J. J. Betancor$^{\ddag,3}$}

\address[$\dag$]{Departamento de Matem\'aticas, Universidad de Las Palmas de Gran Canaria, Edificio de Inform\'atica y Matem\'aticas, Campus de Tafira, 35017, Tafira Baja, Las Palmas de Gran Canaria, Spain.}
\address[$\ddag$]{Departamento de An\'alisis Matem\'atico, Universidad de La Laguna C/. Astrof\'isico Francisco S\'anchez s/n, 38200 - La Laguna, Spain.}
\email[$1$]{isaac.alvarez@ulpgc.es}
\email[$2$]{bbarrios@ull.es}
\email[$3$]{jobetanco@ull.es}
\date{\today}
\maketitle
%%%%%%%%%%%%%%%%%%%%%%%%%%%%%%%%%%%%%%
%%%%%%%%%%%%%%%%%%%%%%%%%%%%%%%%%%%%%%
%%%%%%%%%%%%%%   ABSTRACT   %%%%%%%%%%%%%%%%
%%%%%%%%%%%%%%%%%%%%%%%%%%%%%%%%%%%%%%
%%%%%%%%%%%%%%%%%%%%%%%%%%%%%%%%%%%%%%
\begin{abstract}
In this paper we consider the heat semigroup $\{W_t\}_{t>0}$ defined by the combinatorial Laplacian and two subordinated families of $\{W_t\}_{t>0}$ on homogeneous trees $X$. We characterize the weights $u$ on $X$ for which the pointwise convergence to initial data of the above families holds for every $f\in L^{p}(X,\mu,u)$ with $1\le p<\infty$, where $\mu$ represents the counting measure in $X$ . We prove that this convergence property in $X$ is equivalent to the fact that the maximal operator on $t\in (0,R)$, for some $R>0$, defined by the semigroup is bounded from $L^{p}(X,\mu,u)$ into $L^{p}(X,\mu,v)$ for some weight $v$ on $X$.
\end{abstract}
%%%%%%%%%%%%%%%%%%%%%%%%%%%%%%%%%%%%%%
%%%%%%%%%%%%%%%%%%%%%%%%%%%%%%%%%%%%%%
%%%%%%%%%%%%%%%%%%%%%%%%%%%%%%%%%%%%%%
%%%%%%%%%%%%%%%%%%%%%%%%%%%%%%%%%%%%%%
%%%%%%%%%%%%%%%%%%%%%%%%%%%%%%%%%%%%%%

%%%%%%%%%%%%%%%%%%%%%%%%%%%%%%%%%%%%%%
%%%%%%%%%%%%%%%%%%%%%%%%%%%%%%%%%%%%%%
%%%%%%%%%%  SECTION: INTRODUCTION  %%%%%%%%%%%%
%%%%%%%%%%%%%%%%%%%%%%%%%%%%%%%%%%%%%%
%%%%%%%%%%%%%%%%%%%%%%%%%%%%%%%%%%%%%%

\setcounter{equation}{0}
\section{Introduction}
Along this work we will denote by $X$, a homogeneous tree of degree $q+1$, that is, a connected
graph without loops, in which every vertex has $q+1$ neighbors, with $q\in\mathbb{N}$, $q\geq 1$. We consider on $X$ the natural distance $d$, that is, if $x,\, y\in X$, $x\neq y$, $d(x,y)$ is the number of edges between the vertexes $x$ and $y$ and $d(x,x)=0$, $x\in X$. By $\mu$ we represent the counting measure on $X$. Note that $\mu$ is not doubling because $q\geq 1$. The canonical, also called combinatorial, Laplace operator on $X$ is defined by
\begin{equation}\label{Laplac}
\mathcal{L}(f)(x)=f(x)-\frac{1}{q+1}\sum_{y\in X,\, d(x,y)=1} f(y),\quad x\in X,
\end{equation}
where $f$ is a complex function defined on $X$. When $q=1$ we clearly get that the homogeneous tree $X$ is equal to $\mathbb{Z}$ and $(\mathbb{Z}, d, \mu)$ is a homogeneous type space in the sense of Coifman and Weiss (\cite{CW}). This fact leads to use all the tools of the harmonic analysis in this framework, in particular the standard Calder\'on-Zygmund theory works in this space. We also observe that in this case
$$\mathcal{L}(f)(x)\left(=-\Delta_1(f)(x)\right)=-\frac{1}{2}(f(x+1)-2f(x)+f(x-1)),\quad x\in\mathbb{Z},$$
and the harmonic analysis associated with $\Delta_1$ on $\mathbb{Z}$
has been developed in the last years by several authors (see, for
instance \cite{AdL, AGMP, CGRTV, CRSTV, GKLW, LR} and the references
therein). The operator $\mathcal{L}$ is bounded in $L^{p}(X,\mu)$
for every $1\leq p\leq\infty$ and it is selfadjoint in
$L^{2}(X,\mu)$. Hence $-\mathcal{L}$ generates a $C_0$- semigroup
$\{W_t\}_{t\geq 0}:=\{e^{-t\mathcal{L}}\}_{t\geq 0}$ of operators in
$L^{p}(X,\mu)$ for every $1\leq p<\infty$. This semigroup
$\{W_t\}_{t\geq 0}$ is usually called the heat semigroup associated
with the operator $\mathcal{L}$ because for every $f\in
L^{p}(X,\mu)$, $1\leq p<\infty$, the function $u(t,x):=W_t(f)(x)$,
$t>0$ and $x\in X$, solves the initial value problem
\begin{equation}\label{uno}
\left\{
\begin{array}{ll}
\left(\frac{\partial}{\partial t}+\mathcal{L}\right)u(t,x)=0,\quad t>0,\, x\in X,\\
u(x,0)=f(x),\quad x\in X.
\end{array}
\right.
\end{equation}
We also observe that, for every $t>0$, the operator $W_t$ is invariant under the action of the group of isometries of $(X,d)$, so that, $W_t$ is a convolution operator given by
\begin{equation}\label{calor}
W_t(f)(x)=\int_{X} h_t(x,y)f(y)d\mu(y)=\sum_{y\in X} h_t(x,y)f(y),\quad x\in X.
\end{equation}
The function $h_t$, $t>0$ is usually called the heat kernel associated to $\mathcal{L}$ and we have that $h_t(x,y)=H_t(d(x,y))$, that is, $h_t(x,y)$ depends actually on $d(x,y)$, $x,\, y\in X$. According to \cite{GJK} we have that, for every $t>0$,
$$
H_t(k)=\frac{2e^{-(q+1)t}
}{\pi q^{k/2-1}}\int_0^\pi \frac{e^{2t\sqrt{q}cos(u)}sin(u)(qsin((k+1)u)-sin((k-1)u))}{(q+1)^2-4qcos^2(u)}du,\quad k\in \mathbb{N}\setminus\{0\},
$$
and
$$
H_t(0)=\frac{2q(q+1)e^{-(q+1)t}}{\pi}\int_0^\pi \frac{e^{2t\sqrt{q}cos(u)}sin^2(u)}{(q+1)^2-4qcos^2(u)}du.
$$
\bigskip

We also consider other semigroups of operators defined subordinating with respect to $\{W_t\}_{t>0}.$ For that let us take $0<\alpha<2$ and we define, for every $t>0$,
$$f_{\alpha, t}(s):=
\left\{
\begin{array}{ll}
\displaystyle \frac{1}{2\pi\, i}\int_{a-i\infty}^{a+i\infty}e^{zs-tz^{\frac{\alpha}{2}}}\, dz,\quad s\geq 0,\, a>0,\\
0,\quad s< 0.
\end{array}
\right.
$$
Note that the last integral does not depend on $a>0$. As in \cite[Chapter IX,11]{Yo}, we also define, for every $t>0$ and $f\in L^{p}(X,\mu),\, 1\leq p\leq\infty$,
\begin{align}\label{poisson}
P_t^{\alpha}(f)(x)&:=\int_{0}^{\infty} f_{\alpha,t}(s)W_s(f)(x)\,ds\\
&= \sum_{y\in X}P_t^\alpha(d(x,y))f(y),\quad x\in X\nonumber,
\end{align}
where
$$
P_t^\alpha(k)=\int_0^\infty f_{\alpha,t}(s)H_s(k)ds,\quad k\in \mathbb{N}.
$$
%\textcolor{blue}{whose generator is the discrete version of the fractional Laplacian operator (see \cite[])}.
Thus, for every $1\leq p<\infty$, $\{P_t^{\alpha}\}_{t\geq 0}$ is a $C_0$-semigroup of operators on $L^{p}(X,\mu)$ whose infinitesimal generator is $\mathcal{L}^{\alpha/2}$  (\cite[Theorem 2]{Yo}). Then, for every $f\in L^p(X,\mu)$, the function $v(t,x):=P_t^{\alpha}(f)(x)$, $t>0$, $x\in X$, solves the initial value problem
\begin{equation}\label{dos}
\left\{
\begin{array}{ll}
\left(\frac{\partial}{\partial t}+\mathcal{L}^{\frac{\alpha}{2}}\right)v(t,x)=0,\quad t>0,\, x\in X,\\
v(x,0)=f(x),\quad x\in X.
\end{array}
\right.
\end{equation}
Here $\mathcal{L}^{\frac{\alpha}{2}}$ denotes the $\frac{\alpha}{2}$ power of the operator $\mathcal{L}$ that can be defined, for every $x\in X$, by (see \cite[(5), p. 260]{Yo})
$$\mathcal{L}^{\frac{\alpha}{2}} f(x):=\frac{1}{\Gamma(\frac{\alpha}{2})}\int_{0}^{\infty}{\frac{W_t(f)(x)-f(x)}{t^{1+\frac{\alpha}{2}}}}\, dt,\quad f\in L^p(X,\mu),$$
for every $1\le p<\infty$. Since we have that
$$f_{1, t}(s)=\frac{t}{2\sqrt{\pi}}s^{-3/2}e^{-t^2/4s},\, t,\, s\in (0,\infty),$$
then $\{P_t^1\}_{t\geq 0}$ is the so called Poisson semigroup associated to $\mathcal{L}$. When $\alpha\neq 1$ we do not know a non integral form for the function $f_{\alpha, t}$.

\bigskip
The last subordinate family that we will consider is the following one that was introduced in \cite{ST}. Let $\nu>0$. For every $t>0$ and  $f\in L^p(X,\mu)$, $1\le p\le\infty$, we define
\begin{align}\label{onda}
T_t^{\nu}(f)(x)&:=\frac{t^{2\nu}}{4^{\nu}\Gamma(\nu)}\int_{0}^{\infty} e^{-t^2/4s}W_s(f)(x)\frac{ds}{s^{1+\nu}}\nonumber\\
&=\sum_{y\in X}T_t^\nu(d(x,y))f(y),\quad x\in X,
\end{align}
where
$$
T_t^\nu (k)=\frac{t^{2\nu}}{4^{\nu}\Gamma(\nu)}\int_{0}^{\infty} e^{-t^2/4s}H_s(k)\frac{ds}{s^{1+\nu}},\quad k\in \mathbb{N}.
$$
 As before the family $\{T_t^{\nu}\}_{t\geq 0}$ is a $C_0$- semigroup in $L^{p}(X,\mu)$ when $\nu=1/2$, and, for every $f\in L^{p}(X,\mu)$ the function $w(t,x):=T_t^{\nu}(f)(x)$, $t>0$, $x\in X$, solves the initial value problem
\begin{equation}\label{tres}
\left\{
\begin{array}{ll}
\left(\frac{\partial^2}{\partial t^2}+\frac{1-2\nu}{t}\frac{\partial}{\partial t}+\mathcal{L}\right)w(t,x)=0,\quad t>0,\, x\in X,\\
w(x,0)=f(x),\quad x\in X.
\end{array}
\right.
\end{equation}
We note that $T_t^{1/2}=P_t^{1}$ for every $t>0$.
% and that a function that satisfies \eqref{dos} for $\alpha=1$ clearly fulfills \eqref{tres} for $\nu=1/2$ by taking $\partial_t$ in \eqref{dos} and by using the semigroup property of $\mathcal{L}$.

\bigskip

We can consider \cite{CMS} as the seminal paper in the development of the harmonic analysis associated with the operator $\mathcal{L}$ in homogeneous trees. In this work Cowling, Meda and Setti established some $L^p-L^q$ mapping properties for the heat semigroup defined by $\mathcal{L}$. They also proved $L^p$ boundedness properties for the maximal and Littlewood-Paley functions associated with $\{W_t\}_{t>0}$. The corresponding properties for the Poisson semigroup $\{P_t^{1/2}\}_{t>0}$ were established in \cite{Se}. In the last years the study of harmonic analysis in the homogeneous trees has turned to take great interest. Heat and Poisson semigroups (\cite{KR} and \cite{Sto}), Poincar\'e and Hardy inequalities \cite{BSV}, Hardy and BMO spaces (\cite{ATV1, ATV2, CM}), nondoubling flow measures (\cite{LSTV1, LSTV2}), uncertainty principles (\cite{FJ}), Carleson measures (\cite{CCS}), special multipliers (\cite{CMW}) and maximal functions (\cite{CMS2, GR, NT, ORS}) are some of the topics that have being recently studied in this setting.

\bigskip
For now on we denote $\{S_t\}_{t>0}$ one of the three uniparametric families that we have introduced, that is, $\{W_t\}_{t>0}$, $\{P_t^{\alpha}\}_{t>0}$, $\alpha\in (0,2)$, and $\{T^{\nu}_t\}_{t>0}$, $\nu>0$. It is well known that if $f\in L^{p}(X,\mu)$, $1\leq p<\infty$ then
\begin{equation}\label{star}
\lim_{t\to 0^+} S_t(f)(x)=f(x),\, x\in X.
\end{equation}
Motivated by the results in \cite{GHSTV, HTV} the main objective in this work is to find an optimal weighted Lebesgue space for which \eqref{star} still holds. That is, we want to find the weights $w\in X$, that is, functions $w:X\to (0,\infty)$, such that \eqref{star} is satisfied for every $f\in L^{p}(X,\mu,w)$, $1\leq p<\infty$. Here $L^{p}(X,\mu,w)$ denotes the usual weighted Lebesgue space. The characterization of these weights are related with the maximal function associated to each of the families $\{W_t\}_{t>0}$, $\{P_t^{\alpha}\}_{t>0}$, $\alpha\in (0,2)$ and $\{T^{\nu}_t\}_{t>0}$, $\nu>0$. In fact we have the next result.
\begin{theorem}\label{Teo1}
Let $1\le p<\infty$, $\alpha\in(0,2)$, and $q\in\mathbb{N}$, $q\geq 1$. Assume that $u$ is a weight on $X$. The following assertions are equivalent:
\begin{itemize}
\item [a)] There exist $R>0$ and a weight $v$ on $X$ such that the operator
$$P_{*,R}^{\alpha}(f):=\sup_{0<t<R}|P_t^{\alpha}(f)|,$$
is bounded from $l^{p}(X,\mu,u)$ into $l^{p}(X,\mu,v)$.
%\textcolor{red}{Ponemos el caso $l^{1}\to l^{1,\infty}$?}}
\item [b)]{There exist $R>0$ and a weight $v$ on $X$ such that the operator $P_{\star,R}^{\alpha}$ is bounded from  $l^{p}(X,\mu,u)$ into $l^{p,\infty}(X,\mu,v)$.}
\item [c)]{For every $f\in l^{p}(X,\mu,u)$, $\lim_{t\to 0^+} P_{t}^{\alpha}(f)(x)=f(x)$, $x\in X$.}
\item [d)]{For every $f\in l^{p}(X,\mu,u)$, there exists $x\in X$ such that $\lim_{t\to 0^+} P_{t}^{\alpha}(f)(x)=f(x)$.}
\item [e)]{There exists $R>0$ such that $P_{*,R}^{\alpha}(f)(x)<\infty,\, x\in X$ for every $f\in l^{p}(X,\mu,u)$.}
\item [f)]{There exists $R>0$ such that $P_{R}^{\alpha}(f)(x)<\infty,\, x\in X$ for every $f\in l^{p}(X,\mu,u)$.}
\item [g)] There exists $R>0$ such that, when $1<p<\infty$,
$$\sum_{y\in X}u(y)^{-p'/p}P_{R}^{\alpha}(d(x,y))^{p'}<\infty,\,\,\, x\in X,$$
and, when $p=1$,
$$
\sup_{y\in X}P_R^\alpha(d(x,y))u(y)^{-1}<\infty, \quad x\in X.
$$
%\textcolor{red}{(Ponemos el caso $p=1$, $sup_{y\in X} u(y)^{-1} P_{R}^{\alpha}(d(x,y))<\infty$?)}}
\item [h)] There exist $R>0$ and $x\in X$ satisfying, when $1<p<\infty$,
$$\sum_{y\in X} u(y)^{-p'/p}P_{R}^{\alpha}(d(x,y))^{p'}<\infty,$$
and, when $p=1$,
$$
\sup_{y\in X}P_R^\alpha(d(x,y))u(y)^{-1}<\infty.
$$
\item [i)] There exists $x\in X$ such that, when $1<p<\infty$,
$$\sum_{y\in X}\left(q^{d(x,y)}(1+d(x,y))^{1+\frac{\alpha}{2}}\right)^{-p'}u(y)^{-p'/p}<\infty,$$
and, when $p=1$,
$$
\sup_{y\in X}(q^{d(x,y)}(1+d(x,y))^{1+\frac{\alpha}{2}}u(y))^{-1}<\infty.
$$
\end{itemize}
\end{theorem}
\begin{theorem}\label{Teo2}
Let $1<p<\infty$, $\nu>0$ and $q\in\mathbb{N}$, $q\geq 1$. Assume that $u$ is a weight on $X$. The following assertions are equivalent:
\\a) to h) as in Theorem \ref{Teo1} when the semigroup $\{P_{t}^{\alpha}\}_{t>0}$ is replaced by $\{T_{t}^{\nu}\}_{t>0}$, $\nu>0$.
\begin{itemize}
\item[i)]{There exists $x\in X$ such that, when $1<p<\infty$,
$$\sum_{y\in X}\left(q^{d(x,y)}(1+d(x,y))^{\nu+1}\right)^{-p'}u(y)^{-p'/p}<\infty,$$}
and, when $p=1$,
$$
\sup_{y\in X}(q^{d(x,y)}(1+d(x,y))^{1+\nu}u(y))^{-1}<\infty.
$$
\end{itemize}
\end{theorem}
\begin{theorem}\label{Teo3}
Let $1\le p<\infty$ and $q\in\mathbb{N}$, $q\geq 1$. Assume that $u$ is a weight on $X$. The following assertions are equivalent:
\\a), b), e), f) and g) as in Theorem \ref{Teo1} with the semigroup  $\{W_{t}\}_{t>0}$ instead of $\{P_{t}^{\alpha}\}_{t>0}$.
\begin{itemize}
\item[c)]{For a certain $R_1>0$ there exists $\lim_{t\to R_1}W_t(f)(x)$ for every $f\in l^{p}(X,\mu,u)$ and $x\in X$.}
\item [d)]{For a certain $R_1>0$ and for every $f\in l^{p}(X,\mu,u)$ there exists $x\in X$ for which there exists $\lim_{t\to R_1}W_t(f)(x)$.}
\end{itemize}
Each of the above conditions a)-g) implies that for every $f\in l^{p}(X,\mu,u)$ and $x\in X$, $\lim_{t\to 0^{+}}W_t(f)(x)=f(x).$
\end{theorem}
%
%The previous results have the equivalent ones in the case $q=1$. In fact we obtain the following.
%
%\begin{theorem}\label{Teo1_1}
%Let $1<p<\infty$ and $\alpha\in (0,2)$. Assume that $u$ is a weight on $\mathbb{Z}$. The following assertions are equivalent.
%\\
%Statements a)-h) are the same ones as in Theorem \ref{Teo1} by considering $\{P^{\alpha}_t\}_{t>0}$ with $q=1$ and
%\begin{itemize}
%\item [i)] $\displaystyle\sum_{k\in\mathbb{Z}}(1+|k|)^{-p'\left(1+\frac{\alpha}{2}\right)}u(k)^{-p'/p}<\infty.$
%\end{itemize}
%\end{theorem}
%
%\begin{theorem}\label{Teo2_1}
%Let $1<p<\infty$ and $\nu>0$. Assume that $u$ is a weight on $\mathbb{Z}$. The following assertions are equivalent.
%\\
%Statements a)-h) are the same ones as in Theorem \ref{Teo2} by taking $\{T_t^{\nu}\}_{t>0}$ for $q=1$ and
%\begin{itemize}
%\item [i)] $\displaystyle\sum_{k\in\mathbb{Z}}(1+|k|)^{-p'\left(1+\nu\right)}u(k)^{-p'/p}<\infty.$
%\end{itemize}
%\end{theorem}
%\begin{theorem}\label{Teo3_1}
%Let $1<p<\infty$ and $\nu>0$. Assume that $u$ is a weight on $\mathbb{Z}$. The following assertions are equivalent.
%\\
%Statements a)-g) are the same ones as in Theorem \ref{Teo3} by considering $\{H_t\}_{t>0}$ with $q=1$.
%\end{theorem}

To conclude this introduction we want to mention that harmonic analysis related to flow Laplacian on $X$ has been developed recently in \cite{ATV1, ATV2, BSV, LSTV1, LSTV2}. Main definitions and results about trees with nondoubling flow measures can be found in \cite{LSTV2}. On $X$ the canonical flow measure $\lambda$ is defined by
$$\lambda(x):=q^{l(x)},\, x\in X,$$
where $l$ is the level function on $X$ defined after fixing a reference point and a mythical ancestor. The flow measure $\lambda$ is constant only when $q=1$ and the flow Laplacian (also studied in \cite{HS}), denoted by $\mathbb{L}$ on $X$ and defined by $\mu$ is given by
$$\mathbb{L}(f)(x):=f(x)-\frac{1}{2\sqrt{q}}\sum_{y\sim x}\frac{\lambda(y)}{\lambda(x)}f(y),\, x\in X.$$
We have that
$$\mathbb{L}:=\frac{1}{1-b}\lambda^{-1/2}(\mathcal{L}-bI)\lambda^{1/2},\, \mbox{where } b:=\frac{(\sqrt{q}-1)^2}{q+1}.$$
We denote by $\{\mathbb{W}_{t}\}_{t>0}$ the heat semigroup defined by $\mathbb{L}$ in $L^{2}(X,\lambda)$. Thus, for every $t>0$, the semigroups $W_t$ and $\mathbb{W}_t$ are connected through the expression
$$\mathbb{W}_{t}=\lambda^{-1/2}e^{\frac{bt}{1-b}}W_{\frac{t}{1-b}}\lambda^{1/2}.$$
By using the last equality we can deduce from Theorems \ref{Teo1}-\ref{Teo3} the corresponding proper--ties for $\{\mathbb{W}_t\}_{t>0}$ and the subordinates semigroups associated to it.

\bigskip

\underline {Organization of the work:} In the next three sections we will prove Theorems \ref{Teo1}-\ref{Teo3} and for that some estimations for the heat kernel associated to the Laplacian $\mathcal{L}$ established in \cite{CMS} will be fundamental. The results when $q=1$ can be proved in a similar way that the corresponding ones for $q\geq 2$. However, since some estimates used in those are not valid for the case $X=\mathbb{Z}$, for convenience of the reader, the proofs for $q=1$ are sketched.

Throughout this paper by $C$ and $c$ we always denote positive constants that can change from one line to another.

%%%%%%%%%%%%%%%%%%%%%%%%%%%%%%%%%%%%%%%%%
%%%%%%%%%%%%%%%%%%%%%%%%%%%%%%%%%%%%%%%%%
%%%%%%%%%%%%%   Prueba Teo 1   %%%%%%%%%%%%%
%%%%%%%%%%%%%%%%%%%%%%%%%%%%%%%%%%%%%%%%%
%%%%%%%%%%%%%%%%%%%%%%%%%%%%%%%%%%%%%%%%%
\section{Proof of Theorem \ref{Teo1} for $q\ge 2$}

We consider $q\ge 2$. The properties $a)\Rightarrow b)$,
$c)\Rightarrow d)$, $e)\Rightarrow f)$  and $g)\Rightarrow h)$ are
clear. We denote by
$$\mathcal{F}=\{f:X\to\mathbb{C}\mbox{ the set $\{x\in X:\, f(x)\neq 0\}$ is finite}\}$$
a dense subspace of $l^q(X,\mu,u)$. Then, since for every $f\in\mathcal{F}$, $\displaystyle \lim_{t\to 0^+}P_t^{\alpha}f(x)=f(x)$, by using a standard argument we can see that $b)\Rightarrow c)$.
%%%%%%%%%%%%%%%%%% d) implies e)

Suppose now that $d)$ holds. Let $0\leq f\in l^{p}(X,\mu,u)$. We assume that  $\displaystyle \lim_{t\to 0^+}P_t^{\alpha}(f)(x_0)=f(x_0)$, for some $x_0\in X$. Then, there exists $R_0\in (0,1)$ such that
\begin{equation}\label{lim}
\sup_{0<t<R_0}P_{t}^{\alpha}(f)(x_0)<\infty.
\end{equation}
According to \cite[Theorem 3.1]{Sto} the following property holds for the $\alpha$-stable kernel $P_t^\alpha(k)$, $t>0$ and $k\in \mathbb{N}$,
$$P_{t}^{\alpha}(k)\sim \phi_0(k)tk^{-2-\frac{\alpha}{2}}q^{-k/2},\, t>0,\, k\in\mathbb{N},\, k>t^{2/\alpha}.$$
By using the fact that $q\geq 2$ (see \cite[p. 4282]{CMS} or \cite[(1)]{Sto}),
$$\phi_{0}(k)\sim (k+1)q^{-k/2},\, k\in\mathbb{N},$$
with $\phi_{0}$ being the spherical function (\cite[(1.1)]{CMS}), we get
\begin{equation}\label{A1}
P_t^{\alpha}(k)\sim k^{-1-\frac{\alpha}{2}}tq^{-k},\, t>0,\, k\in\mathbb{N},\, k>t^{2/\alpha}.
\end{equation}
Let us now consider $x_1\in X\setminus\{x_0\}$ and the function $g:X\to\mathbb{R}$ such that $g(x_1)=1$ and $g(x)=0$, $x\in X\setminus\{x_1\}$. We have that
$$
P_t^{\alpha}(g)(x)=\sum_{y\in X}P_{t}^{\alpha}(d(x,y))g(y)=P_{t}^{\alpha}(d(x,x_1)),\, x\in X,\, t>0.
$$
Then, it is clear that
\begin{equation}\label{Con}
\lim_{t\to 0^+}P_t^{\alpha}(k)=0,\, k\in \mathbb{N}\setminus\{0\},\, \lim_{t\to 0^+}P_t^{\alpha}(0)=1.
\end{equation}
By (\ref{Con}) there exists $R_1\in (0,R_0)$ such that $P_{t}^{\alpha}(0)\sim 1$ and $P_t^\alpha(d(x_1,x_0))\le C$, for every $t\in (0, R_1]$.

For every $t>0$, it follows that
\begin{equation}\label{papi}
\begin{split}
P_{t}^{\alpha}(f)(x_1)&=\sum_{y\in X}P_{t}^{\alpha}(d(x_1,y))f(y)=P_{t}^{\alpha}(d(x_1,x_0))f(x_0)\\
&+P_{t}^{\alpha}(0)f(x_1)+\sum_{y\in X\setminus\{x_0,\, x_1\}}\frac{P_{t}^{\alpha}(d(x_1,y))}{P_{t}^{\alpha}(d(x_0,y))}P_{t}^{\alpha}(d(x_0,y))f(y).
\end{split}\end{equation}
Therefore, since without loss of generality we can assume that for $0<t<R_1<1$, $\min\{d(x_0,y),\, d(x_1,y)\}>t^{2/\alpha}$, $y\in X\setminus\{x_0,\, x_1\}$, according to \eqref{A1}, we get that
\begin{equation*}
\begin{split}
\frac{P_{t}^{\alpha}(d(x_1,y))}{P_{t}^{\alpha}(d(x_0,y))}&\sim q^{d(x_0,y)-d(x_1,y)}\left(\frac{d(x_0,y)+1}{d(x_1,y)+1}\right)^{1+\frac{\alpha}{2}}\\
&\leq C q^{d(x_0,x_1)}\left(\frac{d(x_0,x_1)+1+d(x_1,y)}{d(x_1,y)+1}\right)^{1+\frac{\alpha}{2}}\\
&\leq C,\, t\in (0, R_1).
\end{split}
\end{equation*}
Thus, by \eqref{papi} and \eqref{lim} we obtain that
%\begin{equation}\label{neumo}
$$P_{t}^{\alpha}(f)(x_1)\leq C+ \sup_{0<t<R_1}|P_t(f)(x_0)|<\infty,\, 0<t\leq R_1.$$
%\end{equation}
Once we have obtained the previous inequality our next objective is to get
\begin{equation}\label{Madrid}
\sup_{R_1\le t<1}|P_t(f)(x_1)|\leq C|P_{R_1}(f)(x_1)|
\end{equation}
that clearly implies $e)$. It is known that (\cite[p. 89]{Gr})%}
\begin{equation}\label{la_eta}
%\begin{displaymath}
\eta_{t}^{\alpha}(u)\sim\left\{\begin{array}{ll}
\displaystyle{t^{\frac{1}{2-\alpha}}u^{\frac{-(4-\alpha)}{4-2\alpha}}e^{-c_1t^{\frac{2}{2-\alpha}}u^{\frac{-\alpha}{2-\alpha}}}},\, u\leq t^{2/\alpha},\\%[1pc]
\displaystyle{tu^{-1-\frac{\alpha}{2}}e^{-c_1t^{\frac{2}{2-\alpha}}u^{\frac{-\alpha}{2-\alpha}}}},\, u\geq t^{2/\alpha}.
\end{array}\right.
%\end{displaymath}
\end{equation}
where $c_1:=\frac{2-\alpha}{2}\Big(\frac{\alpha}{2}\Big)^{\alpha/(2-\alpha)}$. We fix now $0<a<b<\infty$. We observe that \eqref{Madrid} will be proved if we see that, for some $C>0$,
$$\eta_{t}^{\alpha}(u)\leq C \eta^{\alpha}_{{a}}(u),\, u\in (0,\infty),\, t\in [a,b].$$
To prove the previous inequality let us consider $t\in [a,b]$. By \eqref{la_eta} we have that
%\textcolor{red}{(Check la demostraci\'on de esta desigualdad porque lo hice un poco diferente)}
\begin{itemize}
\item If $u\in (0,a^{2/\alpha})$, then $u\in (0,t^{2/\alpha})$ and
\begin{equation*}
\begin{split}
\eta_{t}^{\alpha}(u)&\leq C\displaystyle{b^{\frac{1}{2-\alpha}}u^{\frac{-(4-\alpha)}{4-2\alpha}}e^{-c_1a^{\frac{2}{2-\alpha}}u^{\frac{-\alpha}{2-\alpha}}}}\\
&\leq C\left(\frac{b}{a}\right)^{\frac{1}{2-\alpha}}\displaystyle{a^{\frac{1}{2-\alpha}}u^{\frac{-(4-\alpha)}{4-2\alpha}}e^{-c_1a^{\frac{2}{2-\alpha}}u^{\frac{-\alpha}{2-\alpha}}}}\\
&\le C \eta_{a}^{\alpha}(u). %\textcolor{red}{\mbox{Check $a\leq t^{2/\alpha}$}.}
\end{split}
\end{equation*}
\item If $u\in (b^{2/\alpha},\infty)$ then $u\in (t^{2/\alpha},\infty)\subseteq (a^{2/\alpha},\infty)$ and
$$
\eta_{t}^{\alpha}(u)\leq C \displaystyle{bu^{-1-\frac{\alpha}{2}}e^{-c_1a^{\frac{2}{2-\alpha}}u^{\frac{-\alpha}{2-\alpha}}}}\le C\eta^{\alpha}_a(u). %\textcolor{red}{\mbox{Check $a\geq t^{2/\alpha}$}.}
$$
\item In the case $u\in[a^{2/\alpha}, t^{2/\alpha}]$ it follows
\begin{equation*}
\begin{split}
\eta_{t}^{\alpha}(u)&\leq C\displaystyle{b^{\frac{1}{2-\alpha}}a^{\frac{-(4-\alpha)}{4-2\alpha}}e^{-c_1a^{\frac{2}{2-\alpha}}u^{\frac{-\alpha}{2-\alpha}}}}\\
&\leq C\displaystyle{b^{\frac{1}{2-\alpha}+\frac{2}{\alpha}(1+\frac{\alpha}{2})}a^{\frac{-(4-\alpha)}{4-2\alpha}-1}au^{-1-\frac{\alpha}{2}}e^{-c_1a^{\frac{2}{2-\alpha}}u^{\frac{-\alpha}{2-\alpha}}}}\\
&\le C\eta^{\alpha}_{a}(u).
\end{split}
\end{equation*}
\item Finally when $u\in[t^{2/\alpha}, b^{2/\alpha}]\subseteq [a^{2/\alpha}, b^{2/\alpha}]$ we also obtain
$$
\eta_{t}^{\alpha}(u)\leq C \displaystyle{bu^{-1-\frac{\alpha}{2}}e^{-c_1a^{\frac{2}{2-\alpha}}u^{\frac{-\alpha}{2-\alpha}}}}\le C\eta^{\alpha}_a(u).
$$
\end{itemize}
%%%%%%%%%%%%%%%%% f) implies g)

Suppose now that $f)$ holds. Let $x\in X$. We consider the mapping $Q_x: l^{p}(X,\mu,u)\to l^1(X,\mu)$ where
$$Q_x(f)(y):=P_{R}^{\alpha}(d(x,y))f(y),\, y\in X.$$
By using the closed graph Theorem we deduce that $Q_x$ is bounded from $l^{p}(X,\mu,u)$ into $l^1(X,\mu)$ and by duality arguments this implies that
$$\sum_{y\in X}P_{R}^{\alpha}(d(x,y))^{p'}u(y)^{-p'/p}<\infty,$$
when $1<p<\infty$, and
$$
\sup_{y\in X}\frac{P_R^\alpha (d(x,y))}{u(y)}<\infty,
$$
when $p=1$, as wanted.
%%%%%%%%%%%%%%%%%%% h) implies i)

If $h)$ holds, then from \eqref{A1} it follows that $i)$ is true.
%%%%%%%%%%%%%%%%%%%i) implies a)

Let us finally assume that $i)$ is true and let us prove that $a)$ is satisfied. By using \eqref{A1}, for a fixed and arbitrary $x\in X$, we obtain
\begin{equation*}
\begin{split}
\sup_{t\in (0,R_1)}|P_{t}^{\alpha}(f)(x)|&\leq\sup_{t\in (0,R_1)} P_{t}^{\alpha}(0)|f(x)|+\sup_{t\in (0, R_1)}\sum_{y\in X\setminus\{x\}}P_{t}^{\alpha}(d(x,y))|f(y)|\\
&\leq C\left(|f(x)|+\sum_{y\in X\setminus\{x\}}(d(x,y)+1)^{-1-\alpha/2}q^{-d(x,y)}|f(y)|\right),
\end{split}
\end{equation*}
for $0<t<R_1$ where $R_1$ is given in the proof of $d)$ implies $e)$. Since for every $x,\, y\in X$,
$$\left(\frac{d(x_0,y)+1}{d(x,y)+1}\right)^{1+\alpha/2}\leq \left(\frac{d(x,y)+d(x_0,x)+1}{d(x,y)+1}\right)^{1+\alpha/2}\leq (d(x_0,x)+1)^{1+\alpha/2},$$
where $x_0\in X$ is the fixed element for which the condition $i)$ is true, then
\begin{equation}\label{todobien}
\begin{split}
&\sup_{t\in (0,R_1)}|P_{t}^{\alpha}(f)(x)|\\
&\leq C\left(|f(x)|+\left(d(x_0,x)+1\right)^{1+\frac{\alpha}{2}}\sum_{y\in X\setminus\{x\}}\frac{|f(y)|}{(d(x_0,y)+1)^{1+\frac{\alpha}{2}}q^{d(x,y)}}\right)\\
&\leq C\left(|f(x)|+\left(d(x_0,x)+1\right)^{1+\frac{\alpha}{2}}q^{d(x_0,x)}\sum_{y\in X\setminus\{x\}}\frac{|f(y)|}{(d(x_0,y)+1)^{1+\frac{\alpha}{2}}q^{d(x_0,y)}}\right).
\end{split}
\end{equation}
Since by H\"older inequality we get when $1<p<\infty$
\begin{equation*}
\begin{split}
S&:=\sum_{y\in X\setminus\{x\}}\frac{|f(y)|}{(d(x_0,y)+1)^{1+\frac{\alpha}{2}}q^{d(x_0,y)}}\\
&\leq\left(\sum_{y\in X\setminus\{x\}}\left(\frac{1}{(d(x_0,y)+1)^{1+\frac{\alpha}{2}}q^{d(x_0,y)}}\right)^{p'}u^{-p'/p}(y)\right)^{1/p'}\left(\sum_{y\in X\setminus\{x\}}|f(y)|^pu(y)\right)^{1/p},
\end{split}
\end{equation*}
and when $p=1$
\begin{equation*}
\begin{split}
S&:=\sum_{y\in X\setminus\{x\}}\frac{|f(y)|}{(d(x_0,y)+1)^{1+\frac{\alpha}{2}}q^{d(x_0,y)}}\\
&\leq \sup_{y\in X\setminus\{x\}}\frac{1}{(d(x_0,y)+1)^{1+\frac{\alpha}{2}}q^{d(x_0,y)}u(y)}\sum_{y\in X\setminus\{x\}}|f(y)|u(y),
\end{split}
\end{equation*}
by the fact that $f\in l^{p}(X,\mu,u)$ and the hypothesis $i)$ it follows that \begin{equation}\label{todobien2}
S<\infty.
\end{equation}
 On the other hand, since we can find $w$ a weight on $X$ such that
$$\sum_{y\in X\setminus\{x\}}\left((d(x_0,y)+1)^{1+\frac{\alpha}{2}}q^{d(x_0,y)}\right)^pw(y)<\infty,$$
then $P_{*,R_1}$ is bounded from $l^p(X,\mu,u)$ into $l^{p}(X,\mu,v)$ as wanted where $v:=\min \{w,u\}$. Indeed given $f\in l^p(X,\mu,u)$
%such that
%$$\left(\sum_{y\in X}|f(y)|^pu(y)\right)^{1/p}<\infty,$$
by \eqref{todobien} and \eqref{todobien2} it follows that
\begin{equation*}
\begin{split}
&\left(\sum_{x\in X}\left(\sup_{t\in (0,R_1)}|P_{t}^{\alpha}(f)(x)|\right)^pv(x)\right)^{1/p}\\
&\leq \left(\sum_{y\in X}|f(y)|^pv(y)\right)^{1/p}+S\left(\sum_{y\in X\setminus\{x\}}\left((d(x_0,y)+1)^{1+\frac{\alpha}{2}}q^{d(x_0,y)}\right)^pv (y)\right)^{1/p}\\
&\le C\left(\sum_{y\in X}|f(y)|^pu(y)\right)^{1/p}.
\end{split}
\end{equation*}%}
Thus, the proof of Theorem \ref{Teo1} is finished when $q\ge 2$.

%%%%%%%%%%%%% Remark
%\begin{remark}
%\textcolor{red}{No lo he escrito porque a\'un no tengo claro d\'onde se usa. Adem\'as debo re-check la última parte del mismo.}
%\end{remark}
%%%%%%%%%%%%%%%%%%%%%%%%%%%%%%%%%%%%%%%%%
%%%%%%%%%%%%%%%%%%%%%%%%%%%%%%%%%%%%%%%%%
%%%%%%%%%%%%%   Prueba Teo 2   %%%%%%%%%%%%%
%%%%%%%%%%%%%%%%%%%%%%%%%%%%%%%%%%%%%%%%%
%%%%%%%%%%%%%%%%%%%%%%%%%%%%%%%%%%%%%%%%%
\section{Proof of Theorem \ref{Teo2} for $q\ge 2$}

We consider $q\ge 2$. Before proving Theorem \ref{Teo2} we establish
the following auxiliary result that plays for the family
$\{T_t^\nu\}_{t>0}$ a similar role as \eqref{A1} for the semigroup
$\{W_{t}^{\alpha}\}_{t>0}$, $\nu>0$ given in \eqref{onda}.
\begin{proposition}\label{Est}
Let $\nu>0$ and $q\in\mathbb{N}$, $q\geq 2$. We have that
\begin{itemize}
\item [a)] For every $t>0$, $k\in\mathbb{N}$ with $k<\nu$, there exists $C=C({q})>0$ such that
$$0\leq T_t^{\nu}(k)\leq C\frac{t^{2k}}{(k+1)^{k+\frac{1}{2}}}\Big(\frac{2e}{q+1}\Big)^{k+1}.$$
\item [b)] For every $t\in (0,1)$, $k\in\mathbb{N}$ with $k\ge\nu$, there exists $C>0$ such that
$$0\le T_{t}^{\nu}(k)\leq C\frac{t^{2\nu}}{k^{\nu+1}q^{k}}.$$
\item [c)] For every $t\in (0,1)$, $k\in\mathbb{N}\setminus\{0\}$ there exists $C>0$ such that
$$T_{t}^{\nu}(k)\geq C\frac{t^{2\nu}}{k^{\nu+1}q^{k}}.$$
\item [d)] For every $t\in (0,1)$, $T_t^{\nu}(0)\sim 1.$
\end{itemize}
\end{proposition}
\begin{proof}
We adapt some ideas developed in \cite[Theorem 3.1]{Sto}. We recall that
\begin{equation}\label{laT}
T_t^{\nu}(k):=\frac{t^{2\nu}}{4^{\nu}\Gamma(\nu)}\int_{0}^{\infty} e^{-t^2/4u}H_u(k)\frac{du}{u^{1+\nu}},\quad k\in\mathbb{N},\, t>0.
\end{equation}
%\textcolor{blue}{where (see for instance \cite{GJK}) $$h_u(k):=...$$}
% demostramos a)
For every $k\in\mathbb{N}\setminus\{0\}$ and $t>0$ by \cite[Proposition 2.5 and p. 4275]{CMS} and \cite[Proposition 3.1]{LSTV1} it follows that
$$
T_t^{\nu}(k)\sim
t^{2\nu}\int_{0}^{\infty}\frac{e^{-\frac{t^2}{4u}-b_2u}}{u^{\nu+2}}\phi_0(k)H_{u\frac{2\sqrt{q}}{q+1}}^\mathbb{Z}(k+1)\,
du,\quad t>0,\quad k\in \mathbb{N},
$$
with $b_2:=1-\frac{2\sqrt{q}}{q+1}\in (0,1)$. Here
$H_t^\mathbb{Z}(k)$, $t>0$ and $k\in \mathbb{N}$, denotes the heat
kernel associated with $X=\mathbb{Z}$ given by
$$H_t^\mathbb{Z}(k)=e^{-t}I_k(t),\quad t>0, \quad k\in \mathbb{N},$$
where $I_\nu$ is the modified Bessel function of the first kind and
order $\nu$. Therefore
\begin{equation*}
\begin{split}
T_t^{\nu}(k)&\leq Ct^{2\nu}\int_{0}^{\infty}e^{-\frac{t^2}{4u}}(k+1)q^{-k/2}\frac{e^{-\left(b_2+\frac{2\sqrt{q}}{q+1}\right)u+\sqrt{(k+1)^2+\frac{4q}{(q+1)^2}u^2}}}{\left(1+(k+1)^2+\frac{4qu^2}{(q+1)^2}\right)^{1/4}}\\
&\cdot \left(\frac{\frac{2\sqrt{q}u}{q+1}}{k+1+\sqrt{(k+1)^2+\frac{4q^2u^2}{(q+1)^2}}}\right)^{k+1}\frac{du}{u^{\nu+2}}\\
&\leq C\frac{t^{2\nu}}{(k+1)^{k+\frac{1}{2}}}\Big(\frac{2e}{q+1}\Big)^{k+1}\int_{0}^{\infty} e^{-b_2u-\frac{t^2}{4u}}u^{k-1-\nu}\, du\\
&\leq C\frac{t^{2\nu}}{(k+1)^{k+\frac{1}{2}}}\Big(\frac{2e}{q+1}\Big)^{k+1}\int_{0}^{\infty} e^{-\frac{t^2}{4u}}u^{k-1-\nu}\, du\\
&\leq C\frac{t^{2k}}{(k+1)^{k+\frac{1}{2}}}\Big(\frac{2e}{q+1}\Big)^{k+1}\int_{0}^{\infty} e^{-z}z^{\nu-k-1}dz\\
&\le C\Gamma(\nu-k)\frac{t^{2k}}{(k+1)^{k+\frac{1}{2}}}\Big(\frac{2e}{q+1}\Big)^{k+1},\quad t>0, \quad k\in \mathbb{N}, \quad k<\nu,
\end{split}
\end{equation*}
where we have used the fact that $b_2>0$ and in the fourth inequality we have done the change of variable $z=t^2/4u$. Thus the conclusion given in $a)$ follows.
% demostramos c)d)
Let us prove the estimates given in $c)$ and $d)$. Indeed, doing similar computations as before, by using the fact that $b_2+\frac{2\sqrt{q}}{q+1}=1$, for every $0<t<1$,  we get that
\begin{equation}\label{grand}
\begin{split}
T_t^{\nu}(k)&\geq Ct^{2\nu}(k+1)q^{-k/2}\int_{0}^{\infty}\frac{e^{-u+\sqrt{(k+1)^2+\frac{4q}{(q+1)^2}u^2}}}{\left(1+(k+1)^2+\frac{4qu^2}{(q+1)^2}\right)^{1/4}}\\
&\cdot \left(\frac{\frac{2\sqrt{q}u}{q+1}}{k+1+\sqrt{(k+1)^2+\frac{4q^2u^2}{(q+1)^2}}}\right)^{k+1}\frac{du}{u^{\nu+2}}\\&\geq Ct^{2\nu}k^{-\nu-\frac{1}{2}}\left(\frac{2}{q+1}\right)^{k}\int_{0}^{\infty}\frac{e^{(k+1)\left(-u+\sqrt{1+\frac{4q}{(q+1)^2}u^2}\right)}u^{k-\nu-1}}{\left(1+\frac{2\sqrt{q}u}{q+1}\right)^{1/2}\left(1+\sqrt{1+\frac{4qu^2}{(q+1)^2}}\right)^{k+1}}\, du\\
&\geq Ct^{2\nu}k^{-\nu-\frac{1}{2}}\left(\frac{2}{q+1}\right)^{k}\int_{0}^{\infty}\frac{e^{(k+1)g(u)}}{u^{\nu+2}\left(1+\frac{2\sqrt{q}u}{q+1}\right)^{1/2}}\, du\\
&:=Ct^{2\nu}k^{-\nu-\frac{1}{2}}\left(\frac{2}{q+1}\right)^{k} I(\nu,q,k).
\end{split}
\end{equation}
where
$$g(u):=\sqrt{1+\frac{4q}{(q+1)^2}u^2}-u+\ln u-\ln \left(1+\sqrt{1+\frac{4qu^2}{(q+1)^2}}\right),\, u>0,$$
is a real function that attains its maximum in $u=(q+1)/(q-1)$ (see \cite[p. 1443]{Sto}). By using the Laplace method we notice that
\begin{equation}\label{Laplac}
I(\nu,q,k)\sim \frac{1}{\sqrt{k}}e^{-k\ln\left(\frac{2q}{q+1}\right)},\quad k\to\infty.
\end{equation}
Therefore, since by similar computations as the ones done in \eqref{grand} we can also obtain that
$$T_t^{\nu}(k)\leq Ct^{2\nu}k^{-\nu-\frac{1}{2}}\left(\frac{2}{q+1}\right)^{k} I(\nu,q,k),\quad k\ge\nu,$$
by \eqref{Laplac} we conclude that
$$T_t^{\nu}(k)\geq C \frac{t^{2\nu}}{k^{\nu+1}q^k},\quad k\to\infty,\, \mbox{uniformly in }\, t\in (0,1),$$
and
$$T_t^{\nu}(k)\leq C \frac{t^{2\nu}}{k^{\nu+1}q^k},\quad k\to\infty,\, k>\nu,\, \mbox{ uniformly in }\, t\in (0,1),$$
as wanted.

Let us finally prove the assertion $d)$. Note that, since from \eqref{laT}
$$T_t^{\nu}(0):=\frac{1}{\Gamma(\nu)}\int_{0}^{\infty} e^{-v}v^{\nu-1}H_{\frac{t^2}{4v}}(0)\, dv,\quad t>0,$$
by using again \cite[Proposition 2.5 and p. 4275]{CMS} and \cite[Proposition 3.1]{LSTV1}, it follows that
$$
T_t^{\nu}(0)\sim\int_{0}^{\infty}e^{-v}v^{\nu-1}e^{-\frac{t^2}{4v}}\frac{\sqrt{q}}{2(q+1)\left(1+\sqrt{1+\left(\frac{t^2\sqrt{q}}{2v(q+1)}\right)^2}\right)}\, dv,\quad t>0.
$$
Then
$$T_t^{\nu}(0)\leq C \int_{0}^{\infty}e^{-v}v^{\nu-1}\, dv\leq C(\Gamma(\nu)),\,\mbox{if } t>0,$$
and
$$T_t^{\nu}(0)\geq C \int_{0}^{\infty}\frac{e^{-v-\frac{1}{4v}}v^{\nu}}{v+\sqrt{1+v^2}}\, dv\geq C,\,\mbox{when }0<t<1.$$
Thus, we conclude that $T_{t}^{\nu}(0)\sim 1$ when $0<t<1$.
\end{proof}
Now we can obtain the

\begin{proofTeo}
To prove $a)\Rightarrow b)\Rightarrow c)\Rightarrow d)$ we can
proceed as in the proof of Theorem \ref{Teo1}. Let us suppose now
that $d)$ holds. We consider $0\leq f\in l^{p}(X,\mu,u)$ and $x_0\in
X$ such that $lim_{t\to 0^+} T_t^{\nu}(f)(x_0)=f(x_0)$. Then there
exists $R_1\in (0,1)$ satisfying
\begin{equation}\label{M}
sup_{0<t\leq R_1} T_t^{\nu}(f)(x_0)<\infty.
\end{equation}
Let us consider $x_1\in X$ and we are going to prove
that
\begin{equation}\label{lae}
sup_{0<t\leq 1} T_t^{\nu}(f)(x_1)<\infty.
\end{equation}
Indeed, first of all we notice that, for every $t\in [R_1,1]$, by
\eqref{onda}, it follows that
%$$T_t^{\nu}(k)=\frac{1}{\Gamma(\nu)}\int_{0}^{\infty}e^{-v}v^{-1+\nu}W_{\frac{t^2}{4v}}(x)dv,\, k\in\mathbb{N},\, t>0,$$
$$T_t^{\nu}(f)(x_1)\leq\frac{1}{4^{\nu}\Gamma(\nu)}\int_{0}^{\infty}\frac{e^{-\frac{R_1^2}{4s}}}{s^{1+\nu}}W_s(f)(x_1)ds\leq R_1^{-2\nu}T_{R_1}^{\nu}(f)(x_1).$$
Secondly, for every $t>0$, we get
\begin{align}\label{B}
T_{t}^{\nu}(f)(x_1)&=\sum_{y\in X}T_{t}^{\nu}(d(x_1,y))f(y)\nonumber\\
 &=\sum_{y\in X,\,min\{d(x_0,y),d(x_1,y)\}<\nu}T_{t}^{\nu}(d(x_1,y))f(y)\nonumber\\
 &\hspace{15mm}+\sum_{y\in
 X,\,min\{d(x_0,y),d(x_1,y)\}\ge\nu}\frac{T_{t}^{\nu}(d(x_1,y))}{T_{t}^{\nu}(d(x_0,y))}T_{t}^{\nu}(d(x_0,y))f(y).
\end{align}
%Moreover, since without loosing of generality for $y\in X$ we can
%assume that $d(x_1,y)>\nu$, \textcolor{red}{$0<\nu<1$ podemos asumir
%esto en $\nu$? Hab\'iamos dicho que s\'i pero ahora no estoy segura.
%Tener en cuenta que la siguiente desigualdad es cierta s\'olo cuando
%$d(x_1,y)>\nu$},
By Proposition \ref{Est} b), it is clear that
\begin{equation}\label{est1}
\begin{split}
\frac{T_{t}^{\nu}(d(x_1,y))}{T_{t}^{\nu}(d(x_0,y))}&\leq C\left(\frac{d(x_0,y)}{d(x_1,y)}\right)^{\nu+1}q^{d(x_0,y)-d(x_1,y)}\\
&\leq C\left(\frac{d(x_0,x_1)+d(x_1,y)}{d(x_1,y)}\right)^{\nu+1}q^{d(x_0,x_1)}\\
&\leq C,\, y\in X, \quad min\{d(x_0,y),d(x_1,y)\}\ge \nu\quad t\in (0,1).%\, d(x_1,y)>\nu
\end{split}
\end{equation}
%and
%$$\frac{T_{R_1}^{\nu}(d(x_1,x_0))}{T_{R_1}^{\nu}(0)}\leq C.$$

By Proposition \ref{Est} (a) and (b), \eqref{M} and \eqref{B} we deduce that there exists $C>0$ such that
%$$
%T_{R_1}^{\nu}(f)(x_1)\leq C\sum_{y\in X,\,
%d(x,y)>\nu}T_{R_1}^{\nu}(d(x_0,y))f(y)\leq C
%T_{R_1}^{\nu}(f)(x_0)<\infty.
%$$
%Then, we obtain
\begin{align}
T_{t}^\nu(f)(x_1)&\le \sum_{y\in X,\,min\{d(x_0,y),d(x_1,y)\}<\nu}T_{t}^{\nu}(d(x_1,y))f(y)\nonumber\\
&\hspace{15mm}+C\sum_{y\in
 X,\,min\{d(x_0,y),d(x_1,y)\}\ge\nu}T_{R_1}^{\nu}(d(x_0,y))f(y)\nonumber\\
 &\le \sum_{y\in X,\,min\{d(x_0,y),d(x_1,y)\}<\nu}T_{t}^{\nu}(d(x_1,y))f(y)+CT_{t}^\nu(f)(x_0)\le C, \quad t\in (0,R_1].
\end{align}
%\textcolor{red}{Nota: Si no podemos asumir $0<\nu<1$ y,
%consecuentemente, $d(x_1,y)>\nu$ lo anterior debe completarse con el
%otro sumando (ver pag 15) cuya finitud se puede justificar de manera
%similar a lo hecho en pag 16}.
We have used that the same sum has a finite number of terms. So that, \eqref{lae} follows and,
consequently also the assertion $e)$ and $f)$.
%%%%%%%%%%%%%%%%% f) implies g)

Let us suppose that $f)$ holds. By proceeding as in the proof of
$f)\Rightarrow g)$ in the previous theorem we deduce that, when $1<p<\infty$,
$$\sum_{y\in X}T_{R}^{\nu}(d(x,y))^{p'}v(y)^{-p'/p}<\infty,\quad x\in X,$$
and, when $p=1$,
$$
\sup_{y\in X}\frac{T_R^\nu (d(x,y))}{u(y)}<\infty, \quad x\in X,
$$
for some $R>0$, and $g)$ is just proved.

The property $g)\Rightarrow h)$ is clear.
%%%%%%%%%%%%%%%%%%% h) implies i)

Suppose now that $h)$ holds. By using Proposition \ref{Est}, $c)$,
we obtain, when $1<p<\infty$,
$$\sum_{y\in X\setminus\{x\}}\left(q^{d(x,y)}d(x,y)^{\nu+1}\right)^{-p'}u(y)^{-p'/p}\leq R^{-2\nu p'} \sum_{y\in X\setminus\{0\}
}(T_{R}^{\nu}(d(0,y)))^{p'}u(y)^{-p'/p}<\infty,$$
and, when $p=1$,
$$
\sup_{y\in X\setminus\{x\}}(q^{d(x,y)}d(x,y)^{\nu+1}u(y))^{-1}\le R^{-2\nu}\sup_{y\in X\setminus\{x\}}\frac{T_R^\nu(d(x,y))}{u(y)}<\infty,
$$
where $R>0$ and $x\in X$ are given
in the assumption $h)$. That is, $i)$ is proved.

%%%%%%%%%%%%%%%%%%% i) implies a)

Let us assume that $i)$ is satisfied. There exists $x_0\in X$ such that, when $1<p<\infty$,
$$\sum_{y\in X}\left(q^{d(x_0,y)}(1+d(x_0,y))^{1+\frac{\alpha}{2}}\right)^{-p'}u(y)^{-p'/p}<\infty,$$
and, when $p=1$,
$$
\sup_{y\in X}(q^{d(x_0,y)}(1+d(x_0,y))^{1+\frac{\alpha}{2}}u(y))^{-1}<\infty.
$$
Let us fix $R>0$. Then, for
every $0<t<R$ and $y\in X$, $y\neq x$ for a fixed but arbitrary
$x\in X$, by (\ref{est1}) we
know that
\begin{equation}\label{B1}
\left(\frac{d(x_0,y)}{d(x,y)}\right)^{\nu+1}q^{d(x_0,y)-d(x,y)}\leq
C(d(x_0,x)+1)^{\nu+1}q^{d(x_0,x)}.
\end{equation}
 Thus, by Proposition
\ref{Est} a), for every $0<t<R$, we get
\begin{equation*}
\begin{split}
|T_t^{\nu}(f)(x)|&\leq |T_t^{\nu}(|f|)(x)|\le C\left(\sum_{y\in X,\, d(x,y)<\nu}\frac{R^{2d(x,y)}}{(d(x,y)+1)^{d(x,y)+\frac{1}{2}}}\left(\frac{2e}{q+1}\right)^{d(x,y)+1} |f(y)|\right.\\
&+ \left. \sum_{y\in X,\, d(x,y)\ge \nu}R^{\nu}d(x,y)^{-\nu-1}q^{-d(x,y)}|f(y)|\right)\\
&\leq C\left(\sum_{y\in X,\, d(x,y)<\nu} |f(y)|+\sum_{y\in X,\, d(x,y)\ge\nu}d(x,y)^{-\nu-1}q^{-d(x,y)}|f(y)|\right)\\
&\leq C\left(|f(x)|+\sum_{y\in
X\setminus\{x\}}d(x,y)^{-\nu-1}q^{-d(x,y)}|f(y)|\right),\quad x\in
X.
\end{split}
\end{equation*}
%\textcolor{red}{Nota: No he quitado los sumando cuando
%$d(x,y)\leq\nu$ porque no veo claro que podamos asumir sin p\'erdida
%de generalidad que $0<\nu<1$. Si fuese as\'i recordemos que, como
%hice en la cuenta d) implica e), se podr\'ian quitar esos sumandos.}
We notice that H\"older inequality and \eqref{B1} lead, when $1<p<\infty$, to
\begin{equation*}
\begin{split}
&\sum_{y\in X\setminus\{x\}} d(x,y)^{-\nu-1}q^{-d(x,y)}|f(y)|\\
&\leq (d(x_0,x)+1)^{\nu+1}q^{d(x_0,x)}\sum_{y\in X\setminus\{x\}} d(x_0,y)^{-\nu-1}q^{-d(x_0,y)}|f(y)|\\
&\leq (d(x_0,x)+1)^{\nu+1}q^{d(x_0,x)}\left(\sum_{y\in
X\setminus\{x\}}
(d(x_0,y)^{\nu+1}q^{d(x_0,y)})^{-p'}u^{}(y)\right)^{\frac{1}{p'}}\left(\sum_{y\in
X}|f(y)|^p\, u(y)\right)^{\frac{1}{p}},
\end{split}
\end{equation*}
and, when $p=1$, to
\begin{equation*}
\begin{split}
&\sum_{y\in X\setminus\{x\}} d(x,y)^{-\nu-1}q^{-d(x,y)}|f(y)|\\
&\leq (d(x_0,x)+1)^{\nu+1}q^{d(x_0,x)}\sum_{y\in X\setminus\{x\}} d(x_0,y)^{-\nu-1}q^{-d(x_0,y)}|f(y)|\\
&\leq (d(x_0,x)+1)^{\nu+1}q^{d(x_0,x)}\sup_{y\in
X\setminus\{x\}}
(d(x_0,y)^{\nu+1}q^{d(x_0,y)}u(y))^{-1}\sum_{y\in
X}|f(y)|\, u(y).
\end{split}
\end{equation*}
Since we can guarantee that
there exists $w$ a weight in $X$ such that
$$\sum_{x\in X}\left((d(x_0,x)+1)^{\nu+1}q^{d(x_0,x)}\right)^p\, w(x)<\infty,$$
by the previous computation we get that $T_{\ast,R}^{\nu}$ is
bounded from $l^{p}(X,\mu,u)$ into $l^{p}(X,\mu,v)$ where $v=\min \{u,w\}$
as wanted.
\end{proofTeo}

%%%%%%%%%%%%%%%%%%%%%%%%%%%%%%%%%%%%%%%%%
%%%%%%%%%%%%%%%%%%%%%%%%%%%%%%%%%%%%%%%%%Palpha
%%%%%%%%%%%%%   Prueba Teo 3   %%%%%%%%%%%%%
%%%%%%%%%%%%%%%%%%%%%%%%%%%%%%%%%%%%%%%%%
%%%%%%%%%%%%%%%%%%%%%%%%%%%%%%%%%%%%%%%%%
\section{Proof of Theorem \ref{Teo3} for $q\ge 2$}
We consider $q\ge 2$. Properties $a)\Rightarrow b)\Rightarrow
c)\Rightarrow d)$ can be proved as in Theorem \ref{Teo1}.
\bigskip
Assume that $d)$ holds. There exists $R_1>0$ such that, for every
$f\in l^p(X,\mu,u)$, the limit
$$\lim_{t\to R_1}H_t(f)(x_0)$$
exists, for some $x_0\in X$.\\

Let $0\leq f\in l^p(X,\mu,u)$ and $x_0\in X$ such that $\lim_{t\to
R_1}W_t(f)(x_0)$ exists. Then, there is a $t_1\in(0,R_1)$, such that
$$\sup_{t\in[t_1,R_1)}W_t(f)(x_0)=A<\infty.$$
Suppose that $R>0$. We have that
\begin{align*}
\frac{1}{t}H_{\frac{t\sqrt{2q}}{q+1}}^\mathbb{Z}(1+d(x,y))&=
\frac{1}{t}e^{-t\frac{\sqrt{2q}}{q+1}}I_{1+d(x,y)}(t\frac{\sqrt{2q}}{q+1})\\
&=\frac{1}{t}e^{-t\frac{\sqrt{2q}}{q+1}}\sum_{k=0}^\infty
\frac{(t\sqrt{2q}/(q+1))^{2k+1+d(x,y)}}{\Gamma(k+1)\Gamma(k+d(x,y)+2)2^{1+d(x,y)+2k}}\\
&=e^{-t\frac{\sqrt{2q}}{q+1}}\Big(\frac{\sqrt{2q}}{q+1}\Big)^{1+d(x,y)}t^{d(x,y)}\sum_{k=0}^\infty
\frac{(t\sqrt{2q}/(q+1))^{2k}}{\Gamma(k+1)\Gamma(k+d(x,y)+2)2^{1+d(x,y)+2k}}\\
&\le
R^{d(x,y)}\Big(\frac{\sqrt{2q}}{q+1}\Big)^{1+d(x,y)}\sum_{k=0}^\infty
\frac{(R\sqrt{2q}/(q+1))^{2k}}{\Gamma(k+1)\Gamma(k+d(x,y)+2)2^{1+d(x,y)+2k}}\\
&\le
\frac{1}{R}e^{R\sqrt{2q}/(q+1)}H_{\frac{R\sqrt{2q}}{q+1}}^\mathbb{Z}(1+d(x,y)),\quad
x,y\in X, \quad 0<t<R.
\end{align*}

 Then, according to \cite[Proposition 2.5]{CMS}, we get
\begin{equation}\label{remark}
\begin{split}
H_t(d(x,y))&\leq C\frac{(q+1)^3\sqrt{q}}{(q-1)^3}\frac{e^{-t}}{t}\phi_0(d(x,y))H_{\frac{t\sqrt{2q}}{q+1}}^\mathbb{Z}(1+d(x,y))\\
&\leq C\frac{(q+1)^3\sqrt{q}}{(q-1)^3}\phi_0(d(x,y))\frac{e^{2R\sqrt{2q}/(q+1)}}{R}H_{\frac{R\sqrt{2q}}{q+1}}^\mathbb{Z}(1+d(x,y))\\
&\leq C\frac{(q+1)^3q}{(q-1)^3}e^{2R(1+\sqrt{2q}/(q+1))}H_R(d(x,y))\\
&=CH_R(d(x,y)),\quad x,y\in X,\quad 0<t<R.
\end{split}
\end{equation}
%\textcolor{red}{Habria que definir quien es $\phi_0(x)$?, que si no me equivoco deberia ser $$\phi_0(x)=\left(1+\frac{q-1}{q+1}\right)q^{-|x|/2}$$}

Thus
$$W_t(f)(x_0)\leq C\max\{A,W_{t_1}(f)(x_0)\},\quad 0<t<R_1.$$
We are going to see that
$$W_{R_0/4}(f)(x)<\infty,\quad x\in X,$$
where $R_0=R_1/2$.\\

Let $x\in X$, by using \cite[Proposition 2.5 and p. 4282]{CMS} and
\cite[Proposition 3.1]{CMS} we get
\begin{equation*}
\begin{split}%comprobar que es r0/4 y r0
\frac{h_{R_0/4}(d(x,y))}{h_{R_0}(d(x_0,y))}&\leq C \frac{(d(x,y)+1)q^{-d(x,y)/2}}{(d(x_0,y)+1)q^{-d(x_0,y)/2}}\left(\frac{1+(1+d(x_0,y))^2+4R_0^2q/(q+1)^2}{1+(1+d(x,y))^2+R_0^2q/4(q+1)^2}\right)^{1/4}\\
&\times\exp\left\{\sqrt{(d(x,y)+1)^2+\frac{R_0^2q}{4(q+1)^2}}-\sqrt{(d(x_0,y)+1)^2+\frac{4R_0^2q}{(q+1)^2}}\right\}\\
&\times\left(\frac{R_0\sqrt{q}/(2(q+1))}{d(x,y)+1+\sqrt{(d(x,y)+1)^2+\frac{R_0^2q}{4(q+1)^2}}}\right)^{1+d(x,y)}\\
&\times\left(\frac{d(x_0,y)+1+\sqrt{(d(x_0,y)+1)^2+4R_0^2q/(q+1)^2}}{2R_0\sqrt{q}/(q+1)}\right)^{1+d(x_0,y)},\quad
y\in X.
\end{split}
\end{equation*}
We divide the estimate in several parts:
\begin{itemize}
\item[]\begin{equation}
\begin{split}
\frac{d(x,y)+1}{d(x_0,y)+1}&\leq\frac{d(x,x_0)+d(x_0,y)+1}{d(x_0,y)+1}\leq
d(x,x_0)+1,\quad y\in X.
\end{split}
\end{equation}
\item[]\begin{equation}
\begin{split}
\frac{q^{-d(x,y)/2}}{q^{-d(x_0,y)/2}}&\leq\frac{q^{-(d(x_0,y)-d(x,x_0))/2}}{q^{-d(x_0,y)/2}}\leq
q^{-d(x,x_0)/2},\quad y\in X.
\end{split}
\end{equation}
\item[]\begin{equation}
\begin{split}
\frac{1+(1+d(x_0,y))^2+\frac{4R_0^2q}{(q+1)^2}}{1+(1+d(x,y))^2+\frac{R_0^2q}{4(q+1)^2}}&\leq \frac{1+(1+d(x_0,x)+d(x,y))^2+\frac{4R_0^2q}{(q+1)^2}}{1+(1+d(x,y))^2+\frac{R_0^2q}{4(q+1)^2}}\\
&\leq C\frac{1+(1+d(x,y))^2+d(x_0,x)^2+\frac{4R_0^2q}{(q+1)^2}}{1+(1+d(x,y))^2+\frac{R_0^2q}{4(q+1)^2}}\\
&\leq C(1+d(x_0,x)^2),\quad y\in X.
\end{split}
\end{equation}
\item[]\begin{equation}
\begin{split}
&\sqrt{(d(x,y)+1)^2+\frac{R_0^2q}{4(q+1)^2}}-\sqrt{(d(x_0,y)+1)^2+\frac{4R_0^2q}{(q+1)^2}}\\
&\leq \frac{(d(x,x_0)+d(x_0,y)+1)^2-(d(x_0,y)+1)^2-\frac{15R_0^2q}{4(q+1)^2}}{\sqrt{(d(x,y)+1)^2+\frac{R_0^2q}{4(q+1)^2}}+\sqrt{(d(x_0,y)+1)^2+\frac{4R_0^2q}{(q+1)^2}}}\\
&=\frac{d(x,x_0)^2+2d(x,x_0)(d(x_0,y)+1)-\frac{15R_0^2q}{4(q+1)^2}}{\sqrt{(d(x,y)+1)^2+\frac{R_0^2q}{4(q+1)^2}}+\sqrt{(d(x_0,y)+1)^2+\frac{4R_0^2q}{(q+1)^2}}}\\
&\leq 2d(x,x_0)+\frac{d(x,x_0)^2(q+1)}{2R_0\sqrt{q}},\quad y\in X.
\end{split}
\end{equation}
\item[]We put
$$
F(y)=A^{1+d(x,y)}B^{1+d(x_0,y)}, \quad y\in X,
$$
where
\begin{equation*}
\begin{split}
A&=\frac{R_0\sqrt{q}/(2(q+1))}{d(x,y)+1+\sqrt{(d(x,y)+1)^2+\frac{R_0^2q}{4(q+1)^2}}}\quad\quad\text{ and }\\
B&=\frac{d(x_0,y)+1+\sqrt{(d(x_0,y)+1)^2+4R_0^2q/(q+1)^2}}{2R_0\sqrt{q}/(q+1)}.
\end{split}
\end{equation*}
We have that

\begin{equation}
\begin{split}
%&\left(\frac{\frac{R_0\sqrt{q}}{4(q+1)}}{d(x,y)+1+\sqrt{(d(x,y)+1)^2+\frac{R_0^2q}{4(q+1)^2}}}\right)^{1+d(x,y)}\\
%&\left(\frac{d(x_0,y)+1+\sqrt{(d(x_0,y)+1)^2+\frac{4R_0^2q}{(q+1)^2}}}{2R_0\sqrt{q}/q+1}\right)^{1+d(x_0,y)}
&F(y)\leq A^{1+d(x_0,y)-d(x_0,x)}B^{1+d(x_0,y)}\\
&\leq C4^{-(1+d(x_0,y))}A^{-d(x_0,x)}\\
&\times\left(\frac{d(x_0,y)+1+\sqrt{(d(x_0,y)+1)^2+4\frac{R_0^2q}{(q+1)^2}}}{d(x,y)+1+\sqrt{(d(x,y)+1)^2+\frac{R_0^2}{2(q+1)}}}\right)^{1+d(x_0,y)}.
\end{split}
\end{equation}
On one hand, we have that $F(y)\leq C$ for $y\in X,$ $d(x_0,y)\leq
4d(x_0,x)$.
%\begin{equation*}
%\begin{split}
%A&=\frac{R_0\sqrt{q}/4(q+1)}{d(x,y)+1+\sqrt{(d(x,y)+1)^2+\frac{R_0^2q}{4(q+1)^2}}}\quad\quad\text{ and }\\
%B&=\frac{d(x_0,y)+1+\sqrt{(d(x_0,y)+1)^2+4R_0^2q/(q+1)^2}}{2R_0\sqrt{q}/q+1}.
%\end{split}
%\end{equation*}

On the other hand we obtain
\begin{equation}
\begin{split}
F(y)&\leq C4^{-(1+d(x_0,y))}%\left(\frac{d(x_0,y)+1+\sqrt{(d(x_0,y)+1)^2+\frac{4R_0^2}{(q+1)^2}}}{|d(x_0,y)-d(x_0,x)|+1+R_0\sqrt{q}/2(q+1)}\right)^{1+d(x_0,y)}\\
%&\left(\frac{1+d(x_0,x)+d(x_0,y)+\sqrt{(1+d(x_0,x)+d(x_0,y)^2+R_0^2q/(q+1)^2}}{R_0\sqrt{q}/2(q+1)}\right)^{d(x_0,x)}\\
\left(\frac{2d(x_0,y)+2+2\frac{R_0\sqrt{q}}{q+1}}{\frac{3}{4}d(x_0,y)+1}\right)^{1+d(x_0,y)}\\
&\times\left(\frac{2(1+\frac{5}{4}d(x_0,y))+\frac{R_0\sqrt{q}}{q+1}}{R_0\sqrt{q}/(2(q+1))}\right)^{d(x_0,x)}\\
&\leq C\Big(\frac{3}{2}\Big)^{-(1+d(x_0,y))}\left(\frac{d(x_0,y)+1+\frac{R_0\sqrt{q}}{q+1}}{d(x_0,y)+\frac{4}{3}}\right)^{1+d(x_0,y)}\\
&\times\left(\frac{2(1+\frac{5}{4}d(x_0,y))+\frac{R_0\sqrt{q}}{q+1}}{R_0\sqrt{q}/(2(q+1))}\right)^{d(x_0,x)}\leq
C,
\end{split}
\end{equation}
for $y\in X$ and $d(x_0,y)\geq 4d(x_0,x)$
%\textcolor{red}{Se explica
%que la constante depende de x y que se contrarrestra por el
%$(3/2)^{\cdot} $?}
\end{itemize}

We conclude that
$$H_{R_0/4}(d(x,y))\leq C H_{R_0}(d(x_0,y)), \quad y\in X.$$
It follows that
$$W_{R_0/4}(f)(x)\leq C W_{R_0}(f)(x_0),\quad x\in X,$$
and $e)$ is established by using the property \eqref{remark}.

It is clear that $e)\Rightarrow f)$.

Assume that $f)$ holds. By using the closed graph theorem and duality as  in Theorem \ref{Teo1} to prove property $g)$, we deduce that $g)$ is proved.\\

The property $g)\Rightarrow h)$ is clear.

 Suppose now that $h)$
holds, that is, for a certain $x_0\in X$ and $R_0>0$, when
$1<p<\infty$, $$ \sum_{y\in
X}P_{R_0}^\alpha(d(x_0,y))^{p'}u(y)^{-p'/p}<\infty ,$$ and, when
$p=1$, $$ \sup_{y\in
X}\frac{P_{R_0}^\alpha(d(x_0,y)}{u(y)}<\infty.$$
From the
estimations in the proof of $d)\Rightarrow e)$, we deduce that there
exists a function $G:\mathbb{N}\to (0,\infty)$ such that

$$H_{R_0/4}(d(x,y))\leq G(d(x_0,x))H_{R_0}(d(x_0,y)),\quad x,y\in X.$$
By using the H\"older's inequality and (\ref{remark}) we get, when
$1<p<\infty$,
\begin{equation*}
\begin{split}
&\sup_{0<t<R_0/4}\Big|\sum_{y\in X}H_t(d(x,y))f(y)\Big|\leq \sum_{y\in X}H_{R_0/4}(d(x,y))|f(y)|\\
&\leq G(d(x_0,x))\sum_{y\in X}H_{R_0}(d(0,y))|f(y)|\\
&\leq G(d(x_0,x))\Big(\sum_{y\in
X}H_{R_0}(d(x_0,y))^{p'}u(y)^{-p'/p}\Big)^{1/p'}\Big(\sum_{y\in
X}|f(y)|^pu(y)\Big)^{1/p},\quad x\in X,
\end{split}
\end{equation*}
and, when $p=1$,
$$
\sup_{0<t<R_0/4}\Big|\sum_{y\in X}H_t(d(x,y))f(y)\Big|\leq
G(d(x_0,x))\sup_{y\in
X}\frac{P_{R_0}^\alpha(d(x_0,y)}{u(y)}\sum_{y\in X}|f(y)|u(y), \quad
x\in X.
$$
 Hence, if $v$ is a weight on $X$ such that
$$\sum_{x\in X}G(d(x,x_0))^pv(x)<\infty$$
%\textcolor{red}{La suma es con respecto a x?}
then $W_{*,R_0/4}$ is bounded from $l^p(X,\mu,u)$ into
$l^p(X,\mu,v)$ and $a)$ is proved

%%%%%%%%%%%%%%%%%%%%%%%%%%%%%%%%%%%%%%%%%
%%%%%%%%%%%%%%%%%%%%%%%%%%%%%%%%%%%%%%%%%Palpha
%%%%%%%%%%%%%   Prueba Teo 1_1   %%%%%%%%%%%%%
%%%%%%%%%%%%%%%%%%%%%%%%%%%%%%%%%%%%%%%%%
%%%%%%%%%%%%%%%%%%%%%%%%%%%%%%%%%%%%%%%%%

\section{Proof of Theorems \ref{Teo1}, \ref{Teo2} and \ref{Teo3} for $q=1$}

In this section we prove the results in the case $q=1$. The homogeneous tree $X$ coincides with $\mathbb{Z}$ and $(\mathbb{Z},d,\mu)$ is a space of homogeneous type in the sense of Coifman and Weiss \cite{CW}.\\
According to \cite[Proposition 3.1]{LSTV1}  we observe
\begin{equation}\label{Eq51}
H_t^Z(k)\sim \frac{1}{(1+|k|+t)^{1/2}}e^{-t(1+\Phi(|k|/t))},\quad
k\in \mathbb{Z}.
\end{equation}
where $\Phi(z)=z\log(z+\sqrt{1+z^2})-\sqrt{1+z^2}$, $z\ge 0$.
%, and
%\begin{equation}\label{Eq52}
%h_t^Z(0)\sim \frac{1}{(1+t)^{1/2}},\quad t>0.
%\end{equation}
In addition we note the following behaviours of $\Phi$, which will be useful in the sequel:
\begin{equation}\label{Eq53}
\Phi(z)\sim \frac{z^2}{2},\quad\text{ as }z\to 0^+,\quad\text{ and }\quad\Phi(z)\sim z\log z,\quad\text{ as }z\to+\infty.
\end{equation}

\begin{proposition}\label{Prop2}
We define $$
\mathbb{P}_t^\alpha(k)=\begin{cases}&t|k|^{-1-\alpha},\quad
k\in\mathbb{Z}\setminus\{0\}\\&1,\quad k=0.\end{cases}$$ Then,
$\mathbb{P}_t^\alpha\sim P_t^\alpha$, $0<t<1$.
\end{proposition}
\begin{proof}
We choose $\beta>1$ and $\alpha/2<\gamma<1$ such that
$\Phi(z)>\gamma z\log z,$ $z>\beta$. We can write
\begin{equation*}
\begin{split}
P_t^\alpha(k)&=\int_0^{|k|/\beta}\eta_t^\alpha(s)H_s^Z(k)ds+\int_{|k|/\beta}^\infty\eta_t^\alpha(s)H^Z_s(k)ds\\
&=I_1^\alpha(t,k)+I_2^\alpha(t,k),\quad t>0,\quad\text{ and }\quad k\in\mathbb{Z}.
\end{split}
\end{equation*}

By using \eqref{Eq51} and \eqref{Eq53} and by applying [(9)
\cite{BSS}] we obtain
\begin{equation*}
\begin{split}
I_2^\alpha(t,k)&\leq C\int_{|k|/\beta}^\infty\eta_t^\alpha(s)\frac{e^{ck^2/2s}}{\sqrt{s}}ds\leq Ct\int_{|k|/\beta}^\infty\frac{e^{-ck^2/2s}}{s^{(\alpha+3)/2}}ds\\
&\leq C \frac{t}{|k|^{\alpha+1}}, \quad t>0,\quad k\in
\mathbb{Z}\setminus\{0\}.
\end{split}
\end{equation*}
On the other hand using once again \eqref{Eq51}, \eqref{Eq53} and
\cite[(9)]{BSS} and by writing $\gamma=a+b$, where $a>\alpha/2$ and
$b>0,$ we get
\begin{equation*}
\begin{split}
I_1^\alpha(t,k)&\leq C\int_0^{|k|/\beta}\eta_t^\alpha(s)\frac{e^{-\gamma|k|\log(|k|/s)}}{\sqrt{|k|}}ds\\
&\leq Ce^{-b|k|\log\beta}\frac{t}{\sqrt{|k|}}\int_0^{|k|/\beta}\frac{1}{s^{1+\alpha/2}}\left(\frac{|k|}{s}\right)^{-a|k|}ds\\
&\leq Ce^{-b|k|\log\beta}\frac{t}{|k|^{1/2+a|k|}}\int_0^{|k|/\beta}s^{a|k|-1-\alpha/2}ds\\
&\leq
Ce^{-b|k|\log\beta}\frac{t}{|k|^{1/2+a|k|}}\left(\frac{|k|}{\beta}\right)^{a|k|-\alpha/2}\leq
C\frac{t}{|k|^{1+\alpha}}, \quad t>0,\quad k\in
\mathbb{Z}\setminus\{0\}.
\end{split}
\end{equation*}
Thus we have proved $P_t^\alpha(k)\leq C \mathbb{P}_t^\alpha(k)$,
$k\in \mathbb{Z}\setminus \{0\}$. To see $P_t^\alpha\geq C
\mathbb{P}_t^\alpha$, we have that there exists $s_0=s_0(\alpha)>0$
such that the following expression holds (see \cite[(10)]{BSS})
\begin{equation}\label{EqA1}
\eta_t^\alpha(s)\geq Cts^{-1-\alpha/2},\quad t>0,\quad\text{ and }\quad s>s_0t^{2/\alpha},
\end{equation}
and by \cite[(7)]{Sto}, for every $a\in(0,1)$ there exists $C>0$
such that
\begin{equation}\label{EqA2}
I_\nu(z)\geq C e^z/z^{1/2},\quad z>\max\{1,a\nu^2\},\quad \nu\geq1,
\end{equation}
where as above $I_\nu$ is the modified Bessel function of the first kind and order $\nu$.\\

By combining \eqref{EqA1} and \eqref{EqA2} we deduce:
\begin{equation*}
\begin{split}
P_t^\alpha(k)&\geq C\int_{k^2/2}^\infty ts^{-1-\alpha/2}e^{-s}\frac{e^s}{s^{1/2}}ds\\
&=Ct\int_{k^2/2}^\infty
s^{-(\alpha+3)/2}ds=C\frac{t}{|k|^{\alpha+1}},\quad t\in(0,1),\quad
k\in \mathbb{Z},\quad |k|\geq\max\{\sqrt{2s_0},\sqrt{2}\}:=k_0.
\end{split}
\end{equation*}
Using \cite[(6)]{Sto} we get
\begin{equation*}
\begin{split}
P_t^\alpha(k)&\geq C\int_{s_0}^\infty ts^{-1-\alpha/2}e^{-s}\frac{e^{\sqrt{s^2+k^2}}}{\sqrt{s+|k|}}\left(\frac{s}{|k|+\sqrt{k^2+s^2}}\right)^|k|ds\\
&\geq Ct\int_{s_0}^\infty s^{-(\alpha+3)/2}\left(\frac{s}{k_0+\sqrt{k_0^2+s^2}}\right)^{k_0}ds\\
&\geq Ct\int_{s_0}^\infty s^{-(\alpha+3)/2}ds=Ct\\
&\geq C\frac{t}{|k|^{\alpha+1}},\quad t\in(0,1),\quad k\in
\mathbb{Z},\quad 0<|k|\leq k_0.
\end{split}
\end{equation*}
Thus we have $P_t^\alpha(k)\sim \mathbb{P}_t^\alpha(k)$, $t\in (0,1)$ and $k\in\mathbb{Z}\setminus\{0\}$.\\
Finally for $k=0$, on one hand, by using \cite[(8)]{BSS} we obtain
\begin{equation*}
\begin{split}
P_t^\alpha(0)&=\int_0^\infty\eta_t(s)h^Z_s(0)ds= t^{-2/\alpha}\int_0^\infty\eta_1(t^{-2/\alpha}s)e^{-2s}I_0(2s)ds\\
&\geq Ct^{-\alpha/2}\int_0^1\eta_1(t^{-2/\alpha}s)ds=C\int_0^{t^{-\alpha/2}}\eta_1(u)du\\
&\geq C\int_0^1\eta_1(u)du,\quad t\in (0,1),
\end{split}
\end{equation*}
and, on the other hand, from \eqref{Eq51} it follows that
\begin{equation*}
\begin{split}
I_2^\alpha(t,0)\leq C\int_0^\infty\eta_t^\alpha(s)ds\leq C,\quad t\in(0,1).
\end{split}
\end{equation*}
We conclude that $P_t^\alpha(0)\sim 1$, $t\in (0,1)$.

\end{proof}

\subsection{Proof of Theorem \ref{Teo1} for $q=1$.}
\bigskip

Properties $a)\Rightarrow b)\Rightarrow c)\Rightarrow d)$ can be proved as in Theorem \ref{Teo1}.
\bigskip
Suppose that $d)$ holds. Let $0\leq f\in l^p(\mathbb{Z},\mu,u)$. We
choose $n_0\in\mathbb{Z}$ such that there exists
$$\lim_{t\to0^+}P^\alpha_t(f)(n_0).$$
 Then, we can find $R_1\in(0,1)$ such that
$$\sup_{0<t\leq R_1}P^{\alpha}_t(f)(n_0)<\infty.$$
We define
$$
\mathbb{P}_t^\alpha(f)(k)=\sum_{n\in \mathbb{Z}}\mathbb{P}_t^\alpha(k-n)f(n),\quad k\in \mathbb{Z}.
$$
Since by Proposition \ref{Prop2} $\mathbb{P}_t^\alpha(k)\sim
P^\alpha_t(k)$, $k\in\mathbb{Z}$ and $t\in (0,1)$, we have that
$\mathbb{P}_{R_1}^\alpha(f)(n_0)<\infty.$ We also have that
$$P^\alpha_t(f)(n_0)\leq C\mathbb{P}_t^\alpha(f)(n_0)\leq \frac{C}{R_1}\mathbb{P}_{R_1}^\alpha(f)(n_0),\quad t\in(R_1,1).$$

Hence, $\sup_{0<t<1}P_t^\alpha(f)(n_0)<\infty$. To see that  $e)$ is true, we need to prove that for any $n\in\mathbb{Z}\setminus\{n_0\}$, $\sup_{0<t<1}P_t^\alpha(f)(n)<\infty$. Thus, let $n_1\in\mathbb{Z}\setminus\{n_0\}$. We can write
\begin{equation*}
\Big|\frac{n_0-k}{n_1-k}\Big|\leq \frac{|n_0-n_1|+|n_1-k|}{|n_1-k|}\leq|n_0-n_1|+1,\quad k\neq n_1.
\end{equation*}
Then,
\begin{equation*}
\begin{split}
P_t^\alpha(f)(n_1)&\leq C\mathbb{P}_t^\alpha(f)(n_1)\\
&=C\left(f(n_1)+tf(n_0)|n_0-n_1|^{-1-\alpha}+\sum_{k\in\mathbb{Z}\setminus\{n_0,n_1\}}\Big(\frac{|n_0-k|}{|n_1-k|}\Big)^{1+\alpha}\mathbb{P}_t^\alpha(n_0-k)f(k)\right)\\
&\leq C\Big( f(n_1)+t\frac{f(n_0)}{|n_0-n_1|^{1+\alpha}}+(|n_0-n_1|+1)^{1+\alpha}\mathbb{P}_t^\alpha(f)(n_0)\Big)\\
&\leq C\Big(
f(n_1)+f(n_0)+(|n_0-n_1|+1)^{1+\alpha}P_t^\alpha(f)(n_0)\Big),\quad
t\in(0,1).
\end{split}
\end{equation*}
Hence $\sup_{t\in(0,1)}P_t^\alpha(f)(n_1)<\infty$ and the property $e)$ is proved.

The property $e)\Rightarrow f)$ is clear.

Suppose that $f)$ holds with $R\in(0,1)$. Then, by using  H\"older's
inequality and duality arguments we obtain when $1<p<\infty$ that
$$\sum_{k\in\mathbb{Z}}P_R^\alpha(|n-k|)^{p'}u(k)^{-p'/p}<\infty,\quad n\in\mathbb{Z},$$
and, when $p=1$, that
$$
\sup_{k\in \mathbb{Z}}\frac{P_R^\alpha(|n-k|)}{u(k)}<\infty,\quad
n\in \mathbb{Z}.
$$ Then, when $1<p<\infty$,
$$\sum_{k\in\mathbb{Z}}(1+|n-k|)^{-(1+\alpha)p'}u(k)^{-p'/p}<\infty,\quad n\in\mathbb{Z},$$
and, when $p=1$,
$$
\sup_{k\in
\mathbb{Z}}\frac{1}{(1+|n-k|)^{1+\alpha}u(k)}<\infty,\quad n\in
\mathbb{Z}.
$$
Thus $g)$ is proved.\\

The property $g)\Rightarrow h)$ is clear.\\

Assume now that $h)$ is true. We observe that
\begin{equation*}
\begin{split}
\frac{1}{1+|n-k|}&=\frac{1+|k|}{1+|n-k|}\frac{1}{1+|k|}\leq\frac{1+|n-k|+|n|}{1+|n-k|}\frac{1}{1+|k|}\\
&\leq\frac{1+|n|}{1+|k|},\quad n,k\in\mathbb{Z}.
\end{split}
\end{equation*}
Let $f:\mathbb{Z}\to[0,\infty)$. By proceeding as above it follows
that
\begin{equation*}
\begin{split}
P_t^\alpha(f)(n)&\leq C\mathbb{P}_t^\alpha(f)(n)\leq C\sum_{k\in\mathbb{Z}}(1+|n-k|)^{-(1+\alpha)}f(k)\\
&\leq  C(1+|n|)^{1+\alpha}\left(\sum_{k\in\mathbb{Z}}(1+|k|)^{-(1+\alpha)p'}u(k)^{-p'/p}\right)^{1/p'}\left(\sum_{k\in\mathbb{Z}}f(k)^pu(k)\right)^{1/p},
\end{split}
\end{equation*}
with $t\in(0,1)$. Then, if $v$ is a weight on $\mathbb{Z}$ such that
$$\sum_{k\in\mathbb{Z}}(1+|k|)^{-(1+\alpha)p}v(k)<\infty,$$
$P_{*,t}^\alpha$  is bounded from $l^p(\mathbb{Z},u)$ into $l^p(\mathbb{Z},v)$.

%%%%%%%%%%%%%%%%%%%%%%%%%%%%%%%%%%%%%%%%%
%%%%%%%%%%%%%%%%%%%%%%%%%%%%%%%%%%%%%%%%%Tnu
%%%%%%%%%%%%%   Prueba Teo 2_1   %%%%%%%%%%%%%
%%%%%%%%%%%%%%%%%%%%%%%%%%%%%%%%%%%%%%%%%
%%%%%%%%%%%%%%%%%%%%%%%%%%%%%%%%%%%%%%%%%
\subsection{Proof of Theorem \ref{Teo2} for $q=1$.} We observe that
\begin{equation*}
\begin{split}
t^{2\nu}\frac{e^{-t^2/4u}}{u^{\nu+1}}&=\frac{t^{2\nu-1}}{u^{\nu-1/2}}\frac{te^{-t^2/4u}}{u^{3/2}}=2\sqrt{\pi}\frac{t^{2\nu-1}}{u^{\nu-1/2}}\eta^{1/2}_t(u),\quad
t,u\in(0,\infty).
\end{split}
\end{equation*}
Thus, by arguing as in the proof of Proposition \ref{Prop2}, we get
$$T_t^\nu(k)\sim\frac{t^{2\nu}}{|k|^{2\nu+1}},\quad t\in (0,1)\quad\text{ and }\quad k\in\mathbb{Z}\setminus\{0\}.$$
By using the last equivalence,  Theorem \ref{Teo2} can be proved as
Theorem \ref{Teo1}
%%%%%%%%%%%%%%%%%%%%%%%%%%%%%%%%%%%%%%%%%
%%%%%%%%%%%%%%%%%%%%%%%%%%%%%%%%%%%%%%%%%calor
%%%%%%%%%%%%%   Prueba Teo 3_1 %%%%%%%%%%%%%
%%%%%%%%%%%%%%%%%%%%%%%%%%%%%%%%%%%%%%%%%
%%%%%%%%%%%%%%%%%%%%%%%%%%%%%%%%%%%%%%%%%
\subsection{Proof of Theorem \ref{Teo3} for $q=1$.} By using \cite[Proposition 5.1]{CMS}
 and by taking into account that
$$\sup_{t\in (0,1]}W_t^Z(f)\sim W_1^Z(f),$$
provided that $f:\mathbb{Z}\to[0,+\infty)$, we deduce that
$$W_{*,1}^Z(f)=\sup_{t\in (0,1]}|W_t^Z(f)|\leq CM(f).$$
Here $M(f)$ denotes the Hardy-Littlewood maximal function. Since
$(\mathbb{Z},d,\mu)$ is a space of homogeneous type we have that
$W_{*,1}^Z(f)$ is bounded from $l^p(\mathbb{Z},\mu,w)$ into itself,
provided that $w$ is in the Muckenhoupt class $A^p(\mathbb{Z})$, see
\cite{HMW}. The class of weights in Theorem \ref{Teo3} is wider than
$A^p(\mathbb{Z})$. The proof of  Theorem \ref{Teo3} follows in a
similar way than the one of Theorem \ref{Teo1} for $q\ge 2$.

%%%%%%%%%%%%%%%%%%%%%%%%%%%%%%%%%%%%%%%%%
%%%%%%%%%%%%%%%%%%%%%%%%%%%%%%%%%%%%%%%%%
%%%%%%%%%%%%%   ACKNOWLEDGEMENTS    %%%%%%%%%%%%%
%%%%%%%%%%%%%%%%%%%%%%%%%%%%%%%%%%%%%%%%%
%%%%%%%%%%%%%%%%%%%%%%%%%%%%%%%%%%%%%%%%%
\bigskip

{\bf Acknowledgements} B. B. was partially supported by Ram\'on y
Cajal fellowship RYC2018-026098-I (Spain) and by grant PID2019-110712GB-I00 from the Spanish Go--vernment. J. B. was partially
supported by grant PID2019-106093GB-I00 from the Spanish Government.
The third author wishes to thank to Professor J. L. Torrea (UAM, Madrid) for
sharing with him some works that show the connection between
convergence of the initial data for evolution equations and the
weight $L^p$ inequalities for the associated maximal semigroup
operators.
%%%%%%%%%%%%%%%%%%%%%%%%%%%%%%%%%%%%%%%%%
%%%%%%%%%%%%%%%%%%%%%%%%%%%%%%%%%%%%%%%%%

%%%%%%%%%%%%%%%%%%%%%%%%%%%%%%%%%%%%%%%%%
%%%%%%%%%%%%%%%%%%%%%%%%%%%%%%%%%%%%%%%%%
%%%%%%%%%%%%%   THE   BIBLIOGRAPHY  %%%%%%%%%%%%%%%
%%%%%%%%%%%%%%%%%%%%%%%%%%%%%%%%%%%%%%%%%
%%%%%%%%%%%%%%%%%%%%%%%%%%%%%%%%%%%%%%%%%

%%%%%%%%%%%%%%%%%%%%%%%%%%%%%%%%%%%%%%%%%
%%%%%%%%%%%%%%%%%%%%%%%%%%%%%%%%%%%%%%%%%
\end{document}